\documentclass[a4paper,UKenglish,cleveref, autoref, thm-restate]{lipics-v2021}

\hideLIPIcs  

\usepackage{amsmath}
\usepackage{amssymb}

\DeclareMathOperator{\tp}{tp}

\newcommand{\eps}{\varepsilon}

\newcommand{\CC}{\mathbb{C}}
\newcommand{\DD}{\mathbb{D}}

\newcommand{\LL}{\mathbb{L}}
\newcommand{\NN}{\mathbb{N}}
\newcommand{\OO}{\mathbb{O}}
\newcommand{\PP}{\mathbb{P}}

\renewcommand{\SS}{\mathbb{S}}
\newcommand{\TT}{\mathbb{T}}
\newcommand{\ZZ}{\mathbb{Z}}

\newcommand{\mcA}{\mathcal{A}}
\newcommand{\mcB}{\mathcal{B}}
\newcommand{\mcC}{\mathcal{C}}

\newcommand{\mcG}{\mathcal{G}}

\newcommand{\mcK}{\mathcal{K}}

\newcommand{\mcP}{\mathcal{P}}

\newcommand{\mfF}{\mathfrak{F}}

\newcommand{\mfJ}{\mathfrak{J}}
\newcommand{\mfK}{\mathfrak{K}}

\title{Local-Order-Invariant Logic on Classes of Bounded Degree} 

\author{Derek {R. Aoki}}{Proof School, USA\and \url{https://derekaoki.com/} }{derekaoki@acm.org}{https://orcid.org/0000-0003-0352-1963}{}

\authorrunning{J. Open Access and J.\,R. Public} 

\Copyright{Jane Open Access and Joan R. Public} 



\acknowledgements{The author would like to thank Noah Schweber for introducing him to the question of capturing epsilon-invariant logic with other invariant logics.}

\nolinenumbers 

\EventEditors{John Q. Open and Joan R. Access}
\EventNoEds{2}
\EventLongTitle{42nd Conference on Very Important Topics (CVIT 2016)}
\EventShortTitle{CVIT 2016}
\EventAcronym{CVIT}
\EventYear{2016}
\EventDate{December 24--27, 2016}
\EventLocation{Little Whinging, United Kingdom}
\EventLogo{}
\SeriesVolume{42}
\ArticleNo{23}

\begin{document}

\maketitle

\begin{abstract}
    Local-order-invariant (first-order) logic is an extension of first-order logic where formulae have access to a ternary local order relation on the Gaifman graph, provided that the truth value does not depend on the specific order relation chosen. Weinstein asked a number of questions about the expressive power of order-invariant and local-order-invariant logics on classes of finite structures of bounded degree, classes of finite structures in general, and classes of locally finite structures. We answer four of his five questions, including showing that local-order-invariant logic collapses to first-order logic on classes of bounded degree. We also investigate epsilon-invariant logic. We show that epsilon-invariant logic collapses to first-order logic on classes of bounded degree by containing it in local-order-invariant logic in this setting, and we give an upper bound for epsilon-invariant logic in terms of local-order-invariant logic on general finite structures. Finally, in the process of proving these theorems, we demonstrate some principles which suggest further directions for showing upper bounds on invariant logics, including an upper bound on epsilon-invariant logic in general.
\end{abstract}

\section{Introduction}

Finite model theory is the study of the descriptions of finite structures, or databases, in formal languages. Generally, questions of concern are about the expressivity or computational complexity of these languages. These languages are known to capture various aspects of computation, such as the seminal result of Fagin \cite{conf/cc/Fagin74} equating the queries in $\mathrm{NP}$ to those expressible in existential second-order logic.

In standard logics like first-order logic, fixed-point and transitive closure logics, or the aforementioned existential second-order logic, one can only refer to the relations of the model itself\textemdash the semantics. One might wish, though, to be able to reference predicates not in the vocabulary of the input databases, such as the order in which the elements of the database are stored, or the successor relation corresponding to that order. A logic must treat isomorphic structures in the same way; to make this use of orderings or successors meaningful, we must say that the query must return the same answer regardless of ordering. The queries which reference orderings but for which changing the ordering does not change the truth value are called order-invariant. 

From here, when we discussion invariant logics, we will specifically refer to invariant first-order logic (as opposed to, for instance, fixed-point logic). By the result of classical model theory known as the Craig Interpolation Theorem, any query defined over arbitrary (possibly infinite) structures which is invariant in this way must be first-order definable. However, the Interpolation Theorem fails in the finite, and this collapse does not occur in the finite either. Thus, there has been a research program investigating the expressivity and computational complexity of such invariant logics on finite structures, of which we will describe the most relevant strands. One strand of investigation has been into finding examples which separate various invariant  logics from first-order logic \cite{DBLP:journals/jsyml/Otto00,DBLP:journals/jsyml/Rossman07}. A second has been into computational and descriptive complexity of invariant logics \cite{DBLP:conf/birthday/BhaskarLW24,DBLP:conf/lics/EngelmannKS12,DBLP:journals/corr/HeuvelKPQRS17}. A third has been into demonstrating collapse in specific instances, such as for successor-invariant logic on classes of bounded degree \cite{DBLP:journals/lmcs/Grange21} or for order-invariant logic on trees, strings, and hollow trees \cite{DBLP:journals/jsyml/BenediktS09,DBLP:conf/csl/GrangeS20}. A final one has been into attempting to apply more general techniques of finite and classical model theory to these logics, such as methods using types \cite{DBLP:journals/corr/BarceloL16} and locality \cite{DBLP:conf/mfcs/GroheS98,DBLP:journals/corr/HarwathS16,lindell2025hanflocalityinvariantelementary}.

A specific logic of interest for databases is epsilon-invariant logic, based on Hilbert's epsilon calculus \cite{Hilbert22} of classical model theory, as discussed in \cite{DBLP:journals/jsyml/Otto00}. This is where, for each first-order formula $\varphi(\vec x,y)$ and tuple $\vec a$, we choose a specific element $b$ of the universe such that $\varphi(\vec a,b)$ holds and fix that, referring to this chosen element in formulae by a term $\epsilon_y(\vec x,y)$. This is of interest in databases because it is in fact equal in expressiveness to the frozen-witness-invariant logic $\mathrm{FO}[witness]^+$ discussed in \cite{DBLP:conf/pods/AbiteboulSV90,DBLP:journals/amai/AbiteboulV91}. Another logic of interest, whose relationship with prior work is instead primarily through locality, is local-order-invariant logic. A local order on a simple graph is a ternary relation $\preceq(x,y,z)$ such that when $x$ is fixed, $\preceq(x,y,z)$ orders the neighbors of $x$. Such local orders have been mentioned in finite model theory before, such as in \cite{DBLP:journals/jsyml/BenediktS09,DBLP:journals/tcs/Courcelle96}, but as a basis for an invariant logic were first introduced in \cite{lindell2025hanflocalityinvariantelementary}. This notion can be easily generalized to arbitrary finite structures by considering instead the Gaifman graph of the structure.

\textbf{Our contribution}:

\begin{enumerate}
    \item In Section 3, we bound local-order-invariant logic on classes of bounded degree, following in the direction of \cite{DBLP:journals/lmcs/Grange21} and answering a question of \cite{lindell2025hanflocalityinvariantelementary}. In Section 4, answering all but one of the remaining questions of Weinstein in \cite{dawar_et_al:DagRep.7.9.1}, we show that both order- and local-order-invariant logics do not collapse to first-order logic on arbitrary finite or arbitrary locally finite structures. We also show that in many cases of concern, classes with Gaifman graphs of bounded diameter, local-order-invariant logic is just as expressive as order-invariant logic. In Section 5, we bound epsilon-invariant logic by local-order-invariant logic on classes of bounded degree, demonstrating that it collapses to first-order logic on such classes.
    \item We believe that this is the first use of Hanf threshold locality to upper bound invariant logics. Previous upper bound results  have followed from actually constructing the invariant structure, such as the explicit construction of a successor in \cite{DBLP:journals/lmcs/Grange21} or of orderings in \cite{DBLP:journals/jsyml/BenediktS09,DBLP:conf/csl/Grange23}. Ours, however, make use of locality theorems in \cite{lindell2025hanflocalityinvariantelementary} which have not been applied concretely until now. We also introduce a method to decompose an invariant logic into simpler ones, applying this to epsilon-invariant logic and indicating that similar decompositions could hold for other invariant logics.
\end{enumerate}

\section{Preliminaries}\label{sec:prelims}

From here, $\sigma$ will be a finite relational vocabulary. We denote by $\mfF_{\sigma}$ the finite $\sigma$ structures. For $\mcA\in\mfF_\sigma$ and $B\subseteq A$, we use the notation $\mcA\upharpoonright\mcB$ to denote the restriction of $\mcA$ to $B$; i.e., the substructure of $\mcA$ with underlying set $B$ such that for a tuple $\vec v$ of elements of $B$ and $R\in \sigma$, $R(\vec{v})$ in $\mcA\upharpoonright B$ if and only if $R(\vec v)$ in $\mcA$. If $\rho\supseteq \sigma$ is a relational vocabulary and $\mcA$ is a $\rho$ structure, we denote by $\mcA\upharpoonright_\sigma$ the reduct of $\mcA$ to $\sigma$, i.e., $\mcA$ with only the relations of $\sigma$.

\begin{definition}
    Let $\mcA$ be a $\sigma$ structure. We define the Gaifman graph of $\mcA$, denoted $\mcG(\mcA)$, such that the vertices of $\mcG(\mcA)$ are the elements of $\mcA$, and there are edges between two vertices $a,a'\in\mcA$ if and only if there exists an $n$-ary relation $R\in\sigma$ and an $n$-tuple $\vec v$ of elements of $\mcA$ such that $R(\vec v)$ in $\mcA$. We now define the following topological notions via this graph.
    \begin{enumerate}
        \item For $a,a'$ in $\mcA$, we say that $\delta^\mcA(a,a')=n$ if the minimum length of a path from $a$ to $a'$ in $\mcG(\mcA)$ is $n$. (This is the standard distance metric on a graph, and is a metric in the usual sense.)
        \item For $a$ in $\mcA$ and $r\geq 0$, $B_r^\mcA(a)$ denotes the set $\{a'\in\mcA:\delta^\mcA(a,a')\leq r\}$, and is called the ball of radius $r$ around $a$ in $\mcA$. We call $S_r^\mcA$, which denotes the set $B_r^\mcA(a)\setminus B_{r-1}^\mcA(a)$ when $r\geq 1$ and $\{a\}$ when $r=0$, the sphere of radius $r$ around $a$ in $\mcA$.
        \item For $a$ in $\mcA$ and $r\geq 0$, we denote by $N_r^\mcA(a)$ the $\sigma\cup \{c\}$ structure $A\upharpoonright B_r^\mcA(a)$ with $c$ a constant symbol which is equal to $a$. We call this the neighborhood of $a$ of radius $r$ in $\mcA$. (This is distinct from the ball of radius $r$ in that the ball is only the underlying set, while the neighborhood inherits the $\sigma$ structure of $\mcA$, and also distinguishes $a$ via $c$.) We say that the $r$-neighborhood type of $a$, denoted $\tp_r^\mcA(a)$, is the class $\{\mcB\in\mfF_\sigma:\mcB\cong N_r^\mcA(a)\}$, i.e., the class of $\sigma$ structures isomorphic to it.
    \end{enumerate}
\end{definition}

The Gaifman graph is an important object of study because it lets us define notions of locality. Locality theorems tell us that it is enough to look at the neighborhoods of the points of a structure to determine first-order properties of the points and the structure. This is partially encapsulated in the following definition and theorem:

\begin{definition}
    Let $r,t\in \omega$. We say that $\mcA$ and $\mcB$ in $\mfF_\sigma$ are Hanf threshold $r,t$-equivalent, denoted $\mcA\approx_{r,t}\mcB$, if $\min(|\{a\in \mcA:\tp_r^\mcA(a)=\tau\}|,t)=\min(|\{b\in \mcB:\tp_r^\mcB(b)=\tau\}|,t)$ for every isomorphism type $\tau$ of $r$ neighborhoods in $\sigma$. We say that $\mfJ\subseteq \mfF_\sigma$ is Hanf $r,t$-threshold local if $\mcA\approx_{r,t}\mcB$ implies that $\mcA\in\mfJ\Leftrightarrow \mcB\in\mfJ$.
\end{definition}

We say that a query $\mfJ\subseteq \mfF_\sigma$ is elementary if it is exactly the models of some first-order formula $\varphi$ in the language of $\sigma$. We have the following theorem, originally proved by Hanf for infinite models in \cite{conf/mt/Hanf65} and later for finite models by Fagin, Stockmeyer, and Vardi in \cite{DBLP:journals/iandc/FaginSV95}:

\begin{theorem}[Fagin-Stockmeyer-Vardi]\label{thm:hanflocalitytheorem}
    Every elementary query on $\mfF_\sigma$ is Hanf $r,t$-threshold local for some $r,t$.
\end{theorem}

Locality results tend not to extend to other logics, as sufficiently powerful ones including monadic second-order logic ($\mathrm{MSO}$) or first-order logic augmented with a transitive closure operator ($\mathrm{FO(TC)}$) can see past the neighborhood of a point. However, the following several definitions will be used to introduce a class of logics which do satisfy the Hanf threshold locality theorem above. Specifically, we will use the framework of \cite{lindell2025hanflocalityinvariantelementary} to introduce the notions of invariant logics.

\begin{definition}
    Let $\mfK$ be a subset of $\mfF_{\sigma}$, $R$ a relation symbol of some arity, and $\PP$ a subset of $\mfF_{\sigma\cup\{R\}}$.
    \begin{enumerate}
        \item For $\mcA'\in \mfF_{\sigma\cup\{R\}}$ and $A\in \mfF_\sigma$, we write $\mcA'\PP \mcA$ if $\mcA'\in \PP$ and $\mcA'\upharpoonright_\sigma=A$ (where $\mcA'\upharpoonright_\sigma$ denotes the reduct of $\mcA'$ to $\sigma$).
        \item We say that $\PP$ is an $R$-presentation scheme for $\mfK$ if for every $\mcA$ in $\mfK$, there exists $\mcA'\in \PP$ such that $\mcA'\PP\mcA$.
        \item We say that $\PP$ is an elementary $R$-presentation scheme for $\mfK$ if it is a first-order definable subset of $\{\mcA'\in \mfF_{\sigma\cup\{R\}}:\mcA'\upharpoonright_\sigma\in\mfK\}$.
    \end{enumerate}
\end{definition}

The first example of presentation schemes to be studied, albeit not in those terms, was order-invariant logic, in which $R$ is binary and a linear order on $\mcA$. This is, notably, an elementary presentation scheme.  Other prevalent examples are traversal-invariant logic, which is when $R$ is binary relation on specifically a simple graph and a graph-theoretic traversal, and successor-invariant logic, when $R$ is binary and a successor relation on a structure corresponding to some linear ordering. The former is elementary, while the latter is not. Each of these have been shown to be strictly stronger in terms of expressive power than first-order logic. The following will be most central to us in this article:

\begin{definition}
    A ternary relation $\preceq(x,y,z)$ is a local order on a simple graph $G$ if for fixed $a\in G$, $\preceq_{a}(y,z)\equiv \preceq(a,y,z)$ orders the $G$-neighbors of $A$.
\end{definition}

As we will see later, this is an elementary presentation scheme. Local orders have some other properties which will be necessary for us to apply theorems of \cite{lindell2025hanflocalityinvariantelementary}, such as the following two:

\begin{definition}
    Let $\PP$ be an $R$-presentation scheme on $\mfK\subseteq \mfF_\sigma$. We say that $\PP$ is neighborhood bounded if there exists $r_\PP$ such that for every $\mcA'$ in $\PP$ and every $a$ in $\mcA'$, we have $B_1^{\mcA'}(a)\subseteq B_{r_\PP}^{\mcA'\upharpoonright_{\sigma}}(a)$.
\end{definition}

\begin{definition}
    Let $\PP$ be an $R$-presentation scheme on $\mfK\subseteq \mfF_\sigma$, with $\mfK$ closed under substructures. We say that $\PP$ is local if for every $\mcA\in \mfK$ and every disjoint $B,C\subseteq \mcA$, the following conditions hold, where we denote $\mcB=\mcA\upharpoonright B$ and $\mcC=\mcA\upharpoonright C$.
    \begin{bracketenumerate}
        \item For every $\mcA'\in\PP$ such that $\mcA'\PP\mcA$, we have $(\mcA'\upharpoonright B)\PP \mcB$.
        \item For every $\mcB',\mcC'\in\PP$ such that $\mcB'\PP\mcB$ and $\mcC'\PP\mcC$, there exists $\mcA'\in \PP$ such that $\mcA'\PP\mcA$, $(\mcA'\upharpoonright B)=\mcB'$, and $(\mcA'\upharpoonright C)=\mcC'$.
    \end{bracketenumerate}
\end{definition}

Condition (1) is called localization, and Condition (2) is called disjoint local amalgamation, but the particulars of these will be important only for one lemma each. Finally, we introduce the actual notion of invariant definability.

\begin{definition}
    Let $\PP$ be an $R$-presentation scheme for a class $\mfK\subseteq \mfF_\sigma$.
    \begin{enumerate}
        \item A first-order sentence $\varphi$ in the language of $\sigma\cup\{R\}$ is $\PP$-invariant over $\mfK$ if and only if for every $\mcA\in\mfK$, $(\exists\mcA'\in\PP)(\mcA'\PP\mcA\wedge\mcA'\models\varphi)\Leftrightarrow (\forall \mcA'\in\PP)(\mcA'\PP\mcA\rightarrow \mcA'\models\varphi)$.
        \item $\mfJ\subseteq \mfK$ is $\PP$-invariant elementary on $\mfK$ if and only if there exists a first-order sentence $\varphi$ which is $\PP$-invariant on $\mfK$ such that for all $\mcA\in\mfK$, $\mcA\in\mfJ$ if and only if $(\exists\mcA'\in\PP)(\mcA'\PP\mcA\wedge\mcA'\models\varphi)$.
    \end{enumerate}
\end{definition}

\section{Collapse of Local-Order-Invariant Logic}

In this section, we introduce the notion of a local order on an arbitrary structure of vocabulary $\sigma$, show that the presentation scheme for local orders is local and neighborhood-bounded, prove via a theorem of \cite{lindell2025hanflocalityinvariantelementary} that local-order-invariant elementary queries are Hanf $r,t$-threshold local for some $r,t\in \omega$, and conclude by demonstrating that Hanf $r,t$-threshold local Boolean queries on classes of bounded degree are elementary.

We begin with said definition of arbitrary local orders.

\begin{definition}
    A ternary relation $\preceq(x,y,z)$ is said to be a local order on $\mcA\in\mfF_\sigma$ if when $x$ is fixed, $\preceq$ linearly orders the neighbors of $x$ in $\mcG(\mcA)$. We define the $\preceq$-presentation scheme $\OO$, where $\preceq$ is ternary, as the class of structures $(\mcA,\preceq)$ where $\mcA$ is in $\mfF_\sigma$ and  $\preceq$ is a local order on $\mcA$. We define the $\leq$-presentation scheme $\LL$ as the class of structures of the form $(\mcA,\leq)$ where $\leq$ is a linear order on $\mcA$.
\end{definition}

It is well-known and clear that $\LL$ is elementary, as was remarked earlier. We mention the existence of the following family of formulae $\varphi_{LO,\psi}(R,y)$, where $\psi(\vec x,y)$ is first-order and $R$ is binary, which says that $R$ is a linear order on the substructure $\mcA\upharpoonright \{a\in\mcA:\psi(\vec x,a)\text{ in }\mcA\}$ for every $\vec x$. The full first-order definition is left for \autoref{proof:elempres}.

With that out of the way, we introduce the theorem of \cite{lindell2025hanflocalityinvariantelementary} we seek to apply.

\begin{theorem}[Lindell-Towsner-Weinstein]\label{thm:localinvhanf}
    Let $\mfF^d_\sigma$ denote the class of elements of $\mfF_\sigma$ with Gaifman graphs of degree at most $d$. Let $\PP$ be a local, neighborhood-bounded elementary $R$-presentation scheme for $\mfF_\sigma^d$. If $\mfJ\subseteq \mfF^d_\sigma$ is an elementary $\PP$-invariant Boolean query on $\mfF_\sigma^d$, then there exist $r,t\in \omega$ such that $\mfJ$ is Hanf $r,t$-threshold local.
\end{theorem}

Our approach now is simply to show that $\OO$ is elementary, local, and neighborhood-bounded, after which showing that Hanf $r,t$-threshold local queries on $\mfF_\sigma^d$ are elementary is sufficient to prove our main theorem. The proofs for these lemmas are below.

We show elementarity, locality, and neighborhood-boundedness for $\mfF_\sigma$ instead of $\mfF_\sigma^d$ as this requires no extra effort, is stronger, and will be used in Section 5.

\begin{lemma}\label{lem:elempres}
    The class $\OO\subseteq \mfF_{\sigma\cup\{\preceq\}}$ of pairs $(\mcA,\preceq)$ such that $\preceq$ is a local order on $\mcA$ is an elementary $\preceq$-presentation scheme for $\mfF_\sigma$.
\end{lemma}

\begin{proof}
    The approach is fairly clear: we construct a formula of the form $\beta_r(x,y)$ for every $r$ such that for fixed $x$, $\beta_r(x,y)$ holds if and only if $y\in B_r^\mcA(x)$; from there, we can say that $\preceq(x,y,z)$ is false when $y$ and $z$ are not both in $S_1^\mcA(x)$ and defines a linear ordering otherwise. The technical details are done in \autoref{proof:elempres}, but we will give an overview. To set up, we define the edge relationship of $\mcG(\mcA)$ in a formula $\eta(x,y)$, inductively define $\beta_r(x,y)$ from $\beta_{r-1}(x,y)$, and define a formula $\phi_r(x,y)$ which says that $y$ is in $S_r^\mcA(x)$. These are done in more generality here than strictly necessary for this lemma because the definability of balls and spheres will be useful frequently. Finally, we simply write out a sentence which says that 
    \begin{enumerate}
        \item $\preceq(x,y,z)$ is restricted to elements satisfying $\phi_1(x,y)$, and
        \item $\preceq(x,y,z)$ restricted to $S_1^\mcA(x)$ is a linear ordering, using the $\varphi_{LO,\psi}(R,x)$ mentioned above.
    \end{enumerate}
\end{proof}

\begin{lemma}\label{lem:localpres}
    $\OO$, as defined previously, is a local and neighborhood-bounded $\preceq$-presentation scheme on $\mfF_\sigma$.
\end{lemma}

\begin{proof}
    We first address Condition (1), locality. For any $\mcA$, $\mcA'$ such that $\mcA'\OO\mcA$, and $B\subseteq \mcA$, the $\preceq$ relation of $\mcA'$ is a local order when restricted to $B$ if and only if for fixed $x$, $\varphi_{LO,\phi_1}(\preceq,x)$ holds in $\mcA'\upharpoonright B$. A linear order on a set is also a linear order on a subset, so this is true. As for Condition (2), disjoint local amalgamation, taking $B$, $C$, $\mcB'$, and $\mcC'$, we can construct the amalgamation $\mcA'$ as follows: Take any local order $\prec$ on $\mcA$ (of course, at least one exists), and then for any $x$ in $B$ or $C$, say that $\preceq(x,y,z)$ if and only if $\preceq(x,y,z)$ in $\mcB'$ or $\mcC'$ respectively. Neighborhood-boundedness is definitional, as $O(x,y,z)$ holding implies that no two of $x$, $y$, and $z$ have distance greater than $2$ in $\mcG(\mcA)$.
\end{proof}

\begin{lemma}\label{lem:defhanf}
    Let $\mfJ$ be a boolean query on $\mfF_\sigma^d$. Then $\mfJ$ is Hanf $r,t$-threshold local if and only if it is elementary.
\end{lemma}

\begin{proof}
    Informally, we wish to exploit the fact that each Hanf $r,t$-threshold local equivalence class is definable, and then use the fact that there are finitely many of these. Let $\mfJ\subseteq \mfF_\sigma^d$ be $r,t$-threshold local. It is known that for any structure $\mcA\in \mfF_\sigma$, there is a formula $\varphi_\mcA$ such that for $\mcB\in\mfF_\sigma$, $\mcB\models \varphi_\mcA$ if and only if $\mcB\cong\mcA$ \cite{DBLP:books/sp/Libkin04}. Because the relation $x\in B_r^\mcA(y)$ is definable in $\mcA$, there is similarly a formula $\varphi_\tau$ for every $r$-neighborhood type $\tau$ such that $(\mcA,x)\models \varphi_\tau$ if and only if $\tp_r^\mcA=\tau$. Because of the degree-boundedness of elements of $\mfF_\sigma^d$, there is $g:\omega^2\to\omega$ such that $N_r^\mcA(a)$ has cardinality at most $g(d,r)$. There are finitely many pointed $\sigma$ structures of cardinality $g(d,r)$ up to isomorphism, so there is $n:\omega^2\to\omega$ such that the isomorphism types of $r$-neighborhoods in elements of $\mfF_\sigma^d$ can be enumerated $\tau_1,\dots \tau_{n(d,r)}$. 
    
    For every $\mcA\in \mfF_\sigma^d$, we define the tuple $c_{r,t}(\mcA)\in \{0,\dots t\}^{n(d,r)}$ as $c_{r,t}(\mcA)_i=\min(|\{a\in \mcA:\tp_r^\mcA(a)=\tau\}|,t)$. By Hanf $r,t$-threshold locality, for $\mcA,\mcB\in \mfF_\sigma^d$, $c_{r,t}(\mcA)=c_{r,t}(\mcB)$ implies $\mcA\in\mfJ\Leftrightarrow \mcB\in\mfJ$. Denote by $T_{\vec m,r,t}$ the set $\{\mcA\in\mfF_\sigma^d:c_{r,t}(\mcA)=\vec m\}$. We know that for every $\vec m$, $T_{\vec m}\subseteq \mfJ$ or $T_{\vec m}\cap \mfJ=\emptyset$. Let $T_\mfJ$ denote $\{\vec m:T_{\vec m\subseteq \mfJ}\}$. We can define $\varphi_{i,\tau}$ which for $i\lneq t$ says that there are exactly $i$ elements of $\mcA$ satisfying $\varphi_\tau$, and for $i=t$ says that there are at least $t$ elements of $\mcA$s satisfying $\varphi_\tau$. The following formula defines $\mfJ$:
    \begin{align*}
        \theta_\mfJ= \bigvee_{\vec m\in T_\mfJ} \bigwedge_{\substack{j\leq t \\ k\leq n(d,r)}} \varphi_{m_j,\tau_{k}}.
    \end{align*}
\end{proof}

Putting these puzzle pieces together is not very difficult, yielding the main theorem of the section:

\begin{theorem}\label{thm:localordercollapse}
    Every elementary $\OO$-invariant Boolean query on $\mfF_\sigma^d$ is first-order definable.
\end{theorem}

\begin{proof}
    By \autoref{lem:elempres} and \autoref{lem:localpres}, we know that \autoref{thm:localinvhanf} applies when $\PP=\OO$. This tells us that elementary $\OO$-invariant Boolean queries on $\mfF_\sigma^d$ are Hanf-local. Finally, by \autoref{lem:defhanf}, we conclude that elementary $\OO$-invariant Boolean queries on $\mfF_\sigma^d$ are first-order definable.
\end{proof}

In fact, this generalizes to queries on any subclass of $\mfF_\sigma^d$, and we will sometimes use this full generality later.

\section{Interesting Corollaries}

In \cite{dawar_et_al:DagRep.7.9.1}, Weinstein asked a number of questions about invariant definability on various classes of structures, including certain infinite ones. We let $\mfK_\sigma$ be the structures of any cardinality with vocabulary $\sigma$. We say that $\mcA\in\mfK_\sigma$ is locally finite if for every $a\in\mcA$, $B_r^\mcA(a)$ is finite. Let $\mfK_\sigma^{<\omega}\subseteq \mfK_\sigma$ denote the class of locally finite $\sigma$ structures. Weinstein's questions were the following: 

\begin{quote}
    Let $\PP$ be one of $\OO$ and $\LL$, and let $\mfK$ be one of $\mfF_\sigma,\mfF_\sigma^d,\mfK_\sigma^{<\omega}$. Is every $\PP$-invariant elementary Boolean query on $\mfK$ elementary on $\mfK$?
\end{quote}

As Weinstein noted, the answer is known when $\PP=\LL$ and $\mfK=\mfF_\sigma$: $\LL$-invariant elementary Boolean queries are strictly more expressive than elementary ones when $\sigma$ contains at least one relation of arity at least 2 \cite{DBLP:books/aw/AbiteboulHV95}, and is not more expressive when $\sigma$ contains only unary relations. The rest of Weinstein's questions have previously been unanswered. In the previous section we resolved the case $\PP=\OO$ and $\mfK=\mfF_\sigma^d$. The remaining four cases will be resolved in this section. We begin with a lemma which (along with its variants) will be used throughout the rest of the article.

\begin{lemma}\label{lem:elemcapt}
    Let $\PP$ and $\PP'$ be $R$- and $S$-presentation schemes on $\mfK\subseteq \mfF_\sigma$ respectively. If there exist first-order formulae $\varphi(\vec x,\vec y)$ and $\varphi'(\vec x)$ in the language of $\sigma\cup \{R\}$ with the length of $\vec y$ the same as the arity of $S$ such that for every $\mcA$ in $\mfK$, $\mcA'$ in $\PP$ such that $\mcA'\PP \mcA$, and tuple $\vec a$ of elements of $\mcA$ of the same length as $\vec x$ such that $\varphi'(\vec a)$ holds in $\mcA'$, $\varphi(\vec a,\vec y)$ defines a relation $S_\varphi$ on $\mcA$ such that $(\mcA,S_\varphi)\PP' \mcA$, then every $\PP'$-invariant elementary Boolean query over $\mfK$ is a $\PP$-invariant elementary Boolean query over $\mfK$. 
\end{lemma}

\begin{proof}
    Roughly what we wish to do is take any $\PP'$-invariant sentence $\theta$ in the language of $\sigma\cup\{R\}$ and replace the occurrences of $R(\vec x)$ with $\varphi(\vec x,\vec y)$. Because $\varphi(\vec x,\vec y)$ defines a $\PP'$ structure on every $\mcA$, then $\mcA'\in \PP$ will model this modified sentence if and only if $\mcA'\upharpoonright_\sigma$ is in the invariant query which $\theta$ defines. The full construction of this modified version of $\theta$ is shown in \autoref{proof:elemcapt}. 
\end{proof}

If $\PP'$ is additionally elementary, then we only need that there exists a tuple $\vec a$ such that the relation $\varphi(\vec a,\vec y)$ induces a $\PP'$ presentation of $\mcA$.

\begin{corollary}
    If $\PP$ and $\PP'$ are $R$-presentation schemes on $\mfK\subseteq \mfF_\sigma$ such that $\PP\subseteq \PP'$, then every $\PP'$-invariant Boolean query is a $\PP$-invariant Boolean query.
\end{corollary}

\autoref{lem:elemcapt} shows, in particular, that for any $\PP'$ and $\PP$ satisfying the hypotheses where $\PP$ is a presentation scheme which adds nothing to the structures (e.g., one in which $R(\vec x)$ is always false), then every $\PP'$-invariant elementary Boolean query on $\mfK$ is elementary. We now detail several applications of this lemma.

\begin{example}
    Our first application is one that has been documented several times. Define a circular successor relation on a structure $\mcA$ to be the graph of a permutation of the elements of $\mcA$ with a single orbit. Define a linear successor relation to be a successor relation induced by some linear order on $\mcA$. Denote by $\SS_1$ the class of structures of the form $(\mcA,S_1)$ with $\mcA$ an element of $\mfF_\sigma$ and $S_1$ a circular successor on $\mcA$, and by $\SS'$ denote the class defined in the same way except with $S'S$ a linear successor on $\mcA$. Note that neither of these are elementary presentation schemes. We claim that the $\SS_1$-invariant elementary queries are exactly the $\SS'$-invariant elementary queries, which we prove by defining each in terms of the other and applying \autoref{lem:elemcapt}. First, for any element $a$ of $\mcA\in\mfF_\sigma$ and any circular successor $S_1$ on $\mcA$, the formula
        \begin{align*}
            \varphi_1(a,x,y)\equiv y\neq a\wedge S_1(x,y))
        \end{align*}
        defines a linear successor, as it is simply taking out the relationship  $S_1(a',a)$ where $a'$ is the $S_1$ predecessor of $a$. Similarly, for any $\mcA$ with linear successor $S'$, and for any tuple $(a,a')$ satisfying
        \begin{align*}
            \varphi_2'(a,a')\equiv (\forall x)\neg S'(a,x)\wedge \neg S'(x,a)
        \end{align*}
        (which says that $a$ is the minimum element of the corresponding linear order and $a'$ is the maximum one), the following formula defines a circular successor:
        \begin{align*}
            \varphi_2(a,a',x,y)\equiv (x=a'\wedge y=a)\vee S'(x,y).
        \end{align*}
        This works because we are simply adding the same relationship we removed in $\varphi_1$.
\end{example}

\begin{example}
    Another class which has been considered is $\SS$, the $S$-presentations for which $S$ is the graph of a permutation (as opposed to one with specifically one orbit). In any structure, the formula 
        \begin{align*}
            \varphi(x,y)\equiv x=x
        \end{align*}
        defines a permutation (the identity). Thus, by \autoref{lem:elemcapt}, the $\SS$-invariant elementary queries are elementary.
\end{example}

\begin{example}\label{ex:applications}
    We now come to an example which has not been stated before, but which has been mentioned. Let $\SS_k$ be the class of $S_k$ presentations of the form $(\mcA,S_k)$, where $S_k$ is a permutation of $\mcA$ with at most $k$ orbits. Let $\DD$ be the class of $D$-presentations where $D$ is a permutation without fixed points. Note that $\DD$ is elementary and $\SS_k$ is not. Then every $\DD$-invariant elementary boolean query is $\SS_k$-invariant elementary. We leave the full constructions of the relevant formulae for \autoref{proof:applications}, but suffice it to say that if every point is a fixed point of $S_k$, then we create a cycle out of them, and otherwise, we take some non-fixed point $b$, its successor $b'$, and insert each of the fixed points into the same cycle as $b$ between $b$ and $b'$.
\end{example}

The following definition will be used to establish a result about the relationships between the expressive powers of $\OO$-invariant elementary and $\LL$-invariant elementary Boolean queries on classes not of bounded degree. In fact, we are roughly defining the opposite of classes of bounded degree:

\begin{definition}
    We say that a class $\mfK\subseteq \mfF_\sigma$ is of bounded diameter if there exists natural $d$ such that for every $\mcA$ in $\mfK$, $\mcG(\mcA)$ has, as a graph, diameter at most $d$.
\end{definition}

We note, in fact, that every known separating example for invariant logics (contained in order-invariant logic) is of bounded diameter. This includes that of Gurevich separating order-invariant logic from first-order logic in \cite{DBLP:books/aw/AbiteboulHV95}, that of Otto separating epsilon-invariant logic in \cite{DBLP:journals/jsyml/Otto00}, and those of Rossman separating successor-invariant logic and first-order logic, successor-invariant and epsilon-invariant logic, and successor-invariant and order-invariant logic in \cite{DBLP:journals/jsyml/Rossman07}. We show that a similar separating example cannot be used to separate order-invariant logic from order-invariant logic in this lemma:

\begin{lemma}\label{lem:bounddiamcapt}
    If $\mfK\subseteq \mfF_\sigma$ is of bounded diameter, then for every $\mfJ\subseteq \mfK$, $\mfJ$ is a $\LL$-invariant elementary Boolean query if and only if it is an $\OO$-invariant elementary Boolean query.
\end{lemma}

\begin{proof}
    We proceed so as to use \autoref{lem:elemcapt} for both directions. That is, we will define a local order from an arbitrary linear order, and we will define with parameters a linear order from an arbitrary local order.
    \begin{enumerate}
        \item[$(\Rightarrow)$] This direction is simpler. Consider the following formula in the language $\sigma\cup\{\leq\}$:
        \begin{align*}
            \varphi_\preceq(x,y,z)\equiv \eta(x,y)\wedge \eta(x,z)\wedge y\leq z.
        \end{align*}
        This says, roughly, that $y$ and $z$ are neighbors of $x$ and that $y$ precedes $z$ in the linear ordering. By similar logic to the proof of \autoref{lem:localpres}, the restriction of a linear ordering to the neighbors of a point yields a linear ordering on the neighbors of the point, so this defines a local ordering on $\mcA\in\mfF_\sigma$ for any $\mcA'=(\mcA,\leq)$ in $\LL$. Thus, \autoref{lem:elemcapt} gives us the first implication.
        \item[$(\Leftarrow)$] We will define inductively, using a local ordering, a linear ordering on $B_r^\mcA(x)$ for every $x$. Define the formula $\lambda_1(x,y,z)\equiv (x=y\wedge \beta_1(z))\vee \preceq(x,y,z)$, which will be our ordering of radius 1. We will now define, from an ordering $\lambda(x,y,z)$ of $B_r^\mcA(x)$, an ordering $\lambda_{r+1}(x,y,z)$ of $B_{r+1}^\mcA(x)$. First, we make $B_r^\mcA(x)$ an initial segment of $\lambda_r+1$. Second, we define a preorder $\lambda'_{r+1}(x,y,z)$ on $S_r^\mcA(x)$ by ordering based on the least elements (with respect to $\lambda_r$) of $B_r^\mcA(x)$ that $y$ and $z$ are connected to in $\mcG(\mcA)$ respectively. The equivalence relation corresponding to this preorder is being connected to the same least element. If for $y$ and $z$ this least element is $b$, then we order $y$ and $z$ in $\lambda_{r+1}(x,y,z)$ via $\lambda_1(b,y,z)$. Finally, we place this ordered $S_{r+1}^\mcA(x)$ entirely after the initial segment $B_r^\mcA(x)$ from earlier.

        The full first-order definitions of these are in \autoref{proof:bounddiamcapt}, but once they are established, we apply \autoref{lem:elemcapt}.
    \end{enumerate}
\end{proof}

Below we describe the specific separating example of $\LL$-invariant elementary definability from elementary definability on arbitrary structures, due to Gurevich but first published in \cite{DBLP:books/aw/AbiteboulHV95}.

\begin{theorem}[Gurevich]\label{thm:evenbool}
    Let $\sigma_{BA}$ denote the vocabulary $\langle \subseteq \rangle$, the vocabulary of Boolean algebras. Let $\mfJ_{2}$ be the Boolean query consisting of elements of $\mfF_{\sigma_{BA}}$ which satisfy the axioms of a Boolean algebra and have an even number of atoms. Then $\mfJ_2$ is $\LL$-invariant elementary but not elementary on $\mfF_{\sigma_{BA}}$.
\end{theorem}

This separating example and our lemma about local-order-invariance on classes of bounded diameter answer three of Weinstein's remaining questions. The next theorem is the case $\PP=\OO$ and $\mfK=\mfF_\sigma$. Weinstein notes that there may be a difference between the $\mfK=\mfF_\sigma$ and $\mfK=\mfF_\sigma^d$ cases, and indeed, there is:

\begin{theorem}\label{thm:unboundeddiamnoncoll}
    There exists a vocabulary $\sigma$ and an $\OO$-invariant elementary Boolean query $\mfJ$ on $\mfF_{\sigma}$ which is not elementary.
\end{theorem}

\begin{proof}
    Let $\mfK_{BA}\subseteq \mfF_{\sigma_{BA}}$ be the collection of finite Boolean algebras. It is well-known that there is a first-order sentence in the language of $\sigma_{BA}$ which defines $\mfK_{BA}$. Let us call this sentence $\varphi_{BA}$. If there were to exist some first-order sentence $\varphi_2$ in the language of $\sigma_{BA}$ which defined $\mfJ_2$ as a Boolean query on $\mfK_{BA}$, then the sentence $\varphi_{BA}\wedge \varphi_2$ would define $\mfJ_2$ on $\mfF_{\sigma_{BA}}$, which is impossible by \autoref{thm:evenbool}. 

    On the other hand, \autoref{thm:evenbool} also tells us that $\mfJ_2$ is $\LL$-invariant elementary as a subset of $\mfF_{\sigma_{BA}}$. Thus, it is definable by some sentence $\varphi_2'$ in the language of $\sigma_{BA}\cup\{\leq\}$ as a subset of $\mfF_{\sigma_{BA}}$, and in particular as a subset of $\mfK_{BA}$. The maximum diameter of a finite Boolean algebra's Gaifman graph is of course 2 (because every element has an edge to $\top$, the maximal element, and $\bot$, the minimal one), so by \autoref{lem:bounddiamcapt}, $\varphi_2'$ is expressible on $\mfK_{BA}$ also as a local-order-invariant sentence in the language of $\sigma_{BA}\cup\{\preceq\}$. Finally, we define the sentence $\varphi_{BA}\wedge \varphi_2'$, which witnesses the $\OO$-invariant elementarity of $\mfJ_2$ as a subset of $\mfF_{\sigma_{BA}}$. This completes the proof.
\end{proof}

Note that all of the definitions of \autoref{sec:prelims} (and about Gaifman graphs in general) can have $\mfF_\sigma^{<\omega}$ replaced by the class $\mfK_\sigma^{<\omega}$ of arbitrary (as opposed to finite) $\sigma$ structures. In the next several results, we will consider these modified definitions instead.

Our next theorem is the case $\PP=\OO$ and $\mfK=\mfK_\sigma^{<\omega}$.

\begin{theorem}
      There exists a vocabulary $\sigma$ and an $\OO$-invariant elementary Boolean query on $\mfK_{\sigma}^{<\omega}$ which is not elementary. 
\end{theorem}

\begin{proof}
    The key observation here is that the only locally finite Boolean algebras are the finite ones, because (again) every element has an edge in the Gaifman graph to $\top$ and $\bot$. Therefore, the argument of \autoref{thm:unboundeddiamnoncoll} goes ahead almost identically: $\varphi_{BA}\wedge \varphi_2'$ referenced in the proof defines precisely the same class when applied to $\mfF_{\sigma_{BA}}$ and $\mfK_{\sigma_{BA}}^{<\omega}$, and the same issues with the existence of $\varphi_2$ occur.
\end{proof}

The final question we answer is a corollary of the previous, which is $\PP=\OO$ and $\mfK=\mfK_\sigma^{<\omega}$.

\begin{corollary}
    There exists a vocabulary $\sigma$ and an $\LL$-invariant elementary Boolean query on $\mfK_{\sigma}^{<\omega}$ which is not elementary. 
\end{corollary}

\begin{proof}
    By the obvious variant of \autoref{lem:elemcapt} in which $\mfF_\sigma$ is replaced by $\mfK_\sigma^{<\omega}$, every $\OO$-invariant elementary Boolean query on $\mfK_{\sigma}^{<\omega}$ is $\LL$-invariant elementary, from which the corollary follows.
\end{proof}

For every one of these separating results, the extant vocabulary $\sigma$ contains only one binary relation. Thus, we can instead say that for any vocabulary $\sigma$ with at least one binary relation, there is a $\PP$-invariant elementary Boolean query $\mfJ$ on $\mfF_\sigma$ which is not elementary. One can naturally extend this to the case where there is at least one relation of arity at least 2, but this leaves open the question of what occurs when every relation is unary. 

\begin{remark}
    As it turns out, order- and local-order-invariant logic are not stronger than first-order logic in this unary case. We sketch the argument below.

Let $\sigma=\langle U_1,\dots U_n\rangle$ be a vocabulary with only unary relations. We show that for every $k$, there exists a $k'$ such that if $\mcA$ and $\mcB$ are equivalent with respect to $\mathrm{FO}[k']$ sentences, denoted $\mcA\equiv_{k'}\mcB$, then they are equivalent with respect to order-invariant sentences of quantifier rank at most $k$, denoted $\mcA\equiv^\leq_k\mcB$. This suffices because it means that quantifier rank $k$ order-invariant logic cannot be more expressive than quantifier rank $k'$ first-order logic. Consider that the $r$-neighborhood of any element $x$ for any $r$ is simply $\{x\}$, and that the radius $r$ neighborhood type is characterized by the set $I_x=\{i\leq n:U_i(x)\}$. There are $2^n$ such sets $I_1,\dots I_{2^n}$. For any $(m_1,\dots m_{2^n})\in \{0,\dots 2^k\}^{2^n}$, there is a first-order sentence which says that $\mcA\in \mfF_\sigma$ has $m_i$ elements such that $I_x=I_i$. For any $\sigma$ structure $\mcA$, define some linear ordering $\leq^\mcA$ such that if $x$ and $y$ are elements of $\mcA$, $I_x=I_i$, $I_y=I_j$, and $i\lneq j$, then $x\lneq y$. A standard Ehrenfeucht-Fra\"iss\'e game argument shows that two pure linear orderings with at least $2^k$ elements are equivalent with respect to $\mathrm{FO}[k]$ \cite{DBLP:books/sp/Libkin04}. Almost exactly the same game argument shows that playing on these ordered $\sigma$ structures is a win for the duplicator as long as the structures are $1,2^k$-threshold equivalent.
\end{remark} 

\section{Epsilon-Invariant Logic}

In this section, we explain epsilon-invariant logic, in which formulae are allowed to make use of a choice function on the powerset of the universe as long as the truth value is independent of the specific choice function. We will give upper bounds on epsilon-invariant logic both in general and on classes of bounded degree, using the results of the previous sections for the former. Finally, we give theorems which suggest further possibilities for upper bounds both on epsilon-invariant logic and more traditional first-order invariant logics.

\subsection{Collapse on Classes of Bounded Degree}

We will give some background on what the epsilon in epsilon-invariant logic is. We define an $\epsilon$ term in the language of $\sigma$ to be an expression of the form $\epsilon_y(\varphi)$, where $\varphi(\vec x,y)$ is a first-order formula in the language of $\sigma$. We define the $\mathrm{FO}+\epsilon$ formulae in the language of $\sigma$ to be the atomic formulae in the language of $\sigma$ whose arguments may include $\epsilon$ terms, as well as the closure of these under Boolean combinations and quantification. We define an $\epsilon_k$ term as an $\epsilon$ term where $\varphi$ has quantifier rank at most $k$ and $\vec y$ has length at most $K$, and $FO+\epsilon_k$ formulae are defined accordingly.

$FO+\epsilon$ formulae are evaluated over structures of the form $(\mcA,\epsilon)$, where $\mcA$ is an element of $\mfF_\sigma$ and $\epsilon$ is a function $\mcP(\mcA)$ to $\mcA$ such that $\epsilon(X)\in X$ for $X\in\mcP(\mcA)$ and $X\neq \emptyset$, where $\mcP(\mcA)$ is the powerset of the domain of $\mcA$. In other words, $\epsilon$ is a choice function on the powerset of $\mcA$. An $\epsilon$ term $\epsilon_y(\varphi(\vec x, y))$ is evaluated as $\epsilon(\{a\in \mcA:\varphi(a,\vec b)\text{ holds in }\mcA\})$ for $\vec b$ a tuple of elements of $\mcA$ of the same length as $\vec y$. 

Now, finally, we come to the invariant half of the phrase epsilon-invariant logic:
\begin{definition}
    Let $\varphi$ be an $FO+\epsilon$ formula. We say that it is $\epsilon$-invariant if, in the same way $\PP$-invariant sentences do not depend on choice of $\PP$ structure, $\varphi$ does not depend on the choice of choice function $\epsilon$. $\epsilon_k$-invariance is defined analogously.
\end{definition}

Just as we used \autoref{lem:elemcapt} for bounding invariant logics defined in terms of presentation schemes, we will use the following lemma to shorten proofs in which we bound epsilon-invariant logic.

\begin{lemma}\label{lem:epscontained}
    Let $\PP$ be an $R$-presentation scheme on $\mfK\subseteq\mfF_\sigma$. If there exist a tuple of variables $\vec z$ and a sentence $\psi(\vec z)$ in the language of $\sigma\cup\{R\}$ such that for every $\mathrm{FO}[k]$ formula $\varphi(\vec x,y)$ in the language of $\sigma$ with $\vec x$ of length $k$, there exists a formula $\Phi_{\varphi}(\vec x,y,\vec z)$ in the language of $\sigma\cup\{R\}$ such that for every $\mcA\in\mfK$, every tuple $\vec a$ of elements of $\mcA$ of the same length as $x$, every tuple $\vec c$ of the same length as $\vec z$ such that $\psi(\vec c)$ holds in $\mcA$, and every $\mcA'\in\PP$ with $\mcA'\PP\mcA$, $\Phi_{\varphi}$ satisfies
    \begin{enumerate}
        \item $(\exists!y)\Phi_{\varphi}(\vec z,y,\vec c)$
        \item and $(\forall y)\Phi_{\varphi}(\vec a,y,\vec c)\rightarrow \varphi(\vec a,y)$,
    \end{enumerate}
    in $\mcA'$, then every $\epsilon_k$-invariant elementary boolean query on $\mfK$ is a $\PP$-invariant elementary boolean query on $\mfK$.
\end{lemma}

\begin{proof}
    The proof of this result is similar to that of \autoref{lem:elemcapt}. Expanding on that, we begin by taking some $\epsilon_k$-invariant sentence and taking its prenex normal form. we can replace the occurrences of a term $\epsilon(\varphi(\vec x,y))$ with a variable $x_\varphi$, and then construct our formula in the language $\sigma\cup \{R\}$. We begin by fixing the parameters $\vec z$, then quantifying over the variables of the actual formula, deciding on some $x_\varphi$ to be $\epsilon(\varphi)$ via $\Phi_\varphi$ and these former two sets of variables, and finally evaluating the quantifier-free part of the modified formula using the decided upon $x_\varphi$'s. The full construction is detailed in \autoref{proof:epscontained}.
\end{proof}

\begin{remark}\label{rem:subeps}
    There is a converse result for $\epsilon$-invariant definability almost identical to \autoref{lem:elemcapt}, except because an $\epsilon_k$ operator is not a first-order relation (it is a relation on relations instead of on elements), it is not subsumed by that lemma. 
\end{remark}

Recall the Hanf threshold theorem (\autoref{thm:hanflocalitytheorem}). Another famous notion of locality is that of Gaifman \cite{conf/lc/Gaifman82}, originally introduced for infinite models. The notion used in the finite is an adaptation of Hella, Libkin, and Nurmonen \cite{lindell2025hanflocalityinvariantelementary}:

\begin{definition}
    Let $Q$ be an unary query on $\mfK\subseteq\mfF_\sigma$, i.e., a boolean query on the set of structures of the form $(\mcA,a)$ with $\mcA\in\mfK$ and $a$ an element of $\mcA$. We say that $Q$ is Gaifman local if there exists an $r$ such that for any such $\mcA$ and pair $a,b$ of elements of $\mcA$, if $\tp_r^\mcA(a)=\tp_r^\mcA(b)$, then $Q(\mcA,a)=Q(\mcA,b)$.
\end{definition}

We have the Gaifman locality theorem below for first-order logic, first shown by Gaifman in the infinite \cite{conf/lc/Gaifman82} and then by Hella, Libkin, and Nurmonen in the finite \cite{DBLP:journals/jsyml/HellaLN99}.

\begin{theorem}[Hella-Libkin-Nurmonen]\label{thm:gaifmanlocality}
    Every elementary unary query is Gaifman local. 
\end{theorem}

(As with Hanf threshold locality, this fails even in weak logics such as $\mathrm{FO(TC)}$ \cite{DBLP:books/sp/Libkin04}, but has in fact been shown to hold for order-invariant logic by \cite{DBLP:conf/mfcs/GroheS98}.) \autoref{thm:gaifmanlocality} is important in our context because it tells us that the arguments of $\epsilon$ are local, which allows us to use local orders to bound epsilon-invariant logic on classes of bounded degree, as in the following lemma.

\begin{lemma}\label{lem:epssubOO}
    Let $\mfK$ be a subset of $\mfF_\sigma^d$, and let $\mfJ\subseteq \mfK$ be an $\epsilon$-invariant elementary boolean query. Then $\mfJ$ is an $\OO$-invariant elementary boolean query on $\mfK$.
\end{lemma}

\begin{proof}
    At a high level, we will express an arbitrary $\eps_k$-invariant elementary boolean query using both parameters and a local order, so as to apply \autoref{lem:epscontained}. Given an $\epsilon_k$-invariant sentence $(Q_1x_1)\dots (Q_nx_n)\varphi(x_1,\dots, x_n)$ with $\varphi$ quantifier-free, by \autoref{thm:gaifmanlocality} there is some $r$ such that each of the $\epsilon$ terms in $\varphi$ depend only on the radius $r$ neighborhoods around $x_1,\dots,x_n$. Our new $\OO$-invariant formula will first set via existential quantifiers enough elements of the structure that no matter what $x_1,\dots, x_n$ are given the $Q_1,\dots ,Q_n$, either all of the elements of neighborhood type $\tau$ are within $B_r(x_i)$ for some $i$ or one of the fixed elements is outside of every $B_r(x_i)$. It is possible to do this because of degree-boundedness. Next, for any $\epsilon$ term in $\varphi$, either there is a witness inside some $B_r(x_i)$, in which case we use the $\lambda_r$ defined in the proof of \autoref{lem:bounddiamcapt} to define a linear ordering on $\bigcup_i B_r(x_i)$ and take the minimum witness with respect to this ordering, or there are only witnesses outside $B_r(x_i)$, in which case we take our witness to be one of the previously fixed elements. The full details can be seen in \autoref{proof:epssubOO}.
\end{proof}

We have shown that $\OO$-invariant elementary Boolean queries are elementary in \autoref{thm:localordercollapse}, so as a corollary of this lemma, epsilon-invariant logic collapses to first-order logic as well.

\begin{theorem}
    Let $\mfK$ be a subset of $\mfF_\sigma^d$, and let $\mfJ\subseteq \mfK$ be an $\epsilon$-invariant elementary boolean query. Then $\mfJ$ is an elementary boolean query on $\mfK$.
\end{theorem}

\subsection{Other Capturings}

In this section, we introduce a natural generalization of $R$-presentations to $\rho$-presentations where $\rho$ is a finite relational vocabulary. We show that the two are equivalent. We also define ways to combine a family $\PP_1,\dots \PP_n$ of $R_1,\dots R_n$-presentations. We show that we can use these combinations to express $\epsilon$-invariant sentences as $\PP$-invariant sentences, the latter of which more finite model theoretic tools are amenable to.

\begin{definition}
    We will define a $\rho$-presentation of $\mfK\subseteq\mfF_\sigma$, where $\rho$ is a finite relational vocabulary, in the natural way (a subset of $\mfF_{\sigma\cup\rho}$ such that every element of $\mfF_\sigma$ has an extension). Elementarity is defined similarly.
\end{definition}

\begin{lemma}\label{lem:multipartequiv}
    Let $\rho$ be a vocabulary $\langle R_1,\dots R_n\rangle$, and let $a_n$ be the arity of $R_n$. Then for every $\rho$-presentation $\PP$ of $\mfK\subseteq\mfF_\sigma$, there is an $R$-presentation $\PP'$ of $\mfK$, where $R$ has arity $a_R=\sum_{i\leq n}a_i$, such that every boolean query $\mfJ\subseteq \mfK$ is $\PP$-invariant elementary if and only if it is $\PP'$-invariant elementary on $\mfK$. Furthermore, this $\PP$ is elementary if and only if $\PP'$ is.
\end{lemma}

\begin{proof}
    We will informally be using, for each $i$, $a_i$ of the inputs of $R$ to code $R_i$. More specifically, denoting $\sum_{j\lneq i}a_j$ by $s_i$, for any $\mcA\in \mfK$ and $R_1,\dots R_n$ such that $(\mcA,R_1,\dots R_n)\in \PP$, we define $R$ on $\mcA$ such that $R(x_1,\dots x_{a_R})$ holds if and only if $R_i(x_{s_i+1},\dots x_{s_{i+1}})$ for every $a$. The proof of \autoref{lem:elemcapt} generalizes immediately to $\rho$-presentation schemes, so we will proceed so as to use this slightly modified lemma. First, for any $\mcA'\in\PP$ we define a $\PP'$ structure on $\mcA'\upharpoonright_\sigma$ via the following formulae:
    \begin{align*}
        \varphi_R(x_1,\dots x_{a_R})\equiv \bigwedge_{i\leq n} R_i(x_{s_i+1},\dots x_{s_{i+1}}).
    \end{align*}
    Similarly, we can define a $\PP$ structure from a $\PP'$ structure, via the following family of formulae:
    \begin{align*}
        \varphi_{R_i}(x_1,\dots x_{a_i})\equiv (\forall x_1,\dots x_{s_i},x_{s_{i+1}+1}\dots x_{a_R}) R(x_1,\dots x_{a_R}).
    \end{align*}

    The elementarity is similarly easy to show by simply defining the relevant formulae. Say that $\PP$ is defined by $\phi$ in the language of $\sigma\cup\rho$. Then in a construction almost identical to that of $\theta'''$ of \autoref{lem:elemcapt}, we can replace the occurrences of each relation in $\rho$ with the corresponding formula $\varphi_{R_i}$. A similar approach works in reverse to show that elementarity of $\PP'$ implies elementarity of $\PP$.
\end{proof}

\begin{remark}\label{rem:rhosubeps}
    A lemma analogous to \autoref{lem:elemcapt} can also be established for $\rho$-presentation schemes, and indeed for $\rho$-presentation schemes and $\epsilon$-invariant elementarity as in \autoref{rem:subeps}.
\end{remark}

The following definition introduces a method by which to combine several $R_i$-presentation schemes into a single $\rho$-presentation.

\begin{definition}
    If $\PP_1,\dots \PP_n$ are respectively $R_1,\dots R_n$-presentation schemes of $\mfK\subseteq\mfF_\sigma$, we define their join $\PP_1\oplus \dots \oplus \PP_n$ to be the $\rho$-presentation scheme such that $(\mcA,R_1,\dots R_n)$ is in $\PP$ if and only if $(\mcA,R_i)$ is in $\PP_i$ for every $i$.
\end{definition}

If $\PP_1,\dots \PP_n$ are respectively $R_1,\dots R_n$-presentation schemes of $\mfK\subseteq\mfF_\sigma$, then $\PP_1\oplus\dots\PP_n$ is an elementary presentation scheme if and only if each $\PP_i$ is elementary, by the same logic as the proof of \autoref{lem:multipartequiv}.

We introduce several classes of presentation schemes whose join will be closely related to the expressivity of $\epsilon$-invariant elementary and $\OO$-invariant elementary Boolean queries.

\begin{definition}
    Let a ternary relation $C(x,y,z)$ be a local circular successor if when $x$ in $\mcA\in \mfF_\sigma$ is fixed, $C$ is a circular successor on the neighbors of $x$ in $\mcG(\mcA)$. Let $\CC$ be the class of $C$-presentations which are local successors.

    Let a binary relation $N(x,y)$ be a neighbor selector if for every $x$ in $\mcA\in\mfF_\sigma$, there exists exactly one $y$ such that $N(x,y)$, and $y$ is a neighbor of $x$ in $\mcG(\mcA)$. Let $\NN$ be the class of $N$-presentations which are neighbor selectors.
\end{definition}

Specifically, we prove that $\CC\oplus\NN$-,$\epsilon$, and $\OO$-invariant elementary expressibility are the same. 

\begin{proposition}
    If $\mfJ$ is a boolean query on $\mfK\subseteq \mfF_\sigma^d$, the following are equivalent:
    \begin{bracketenumerate}
        \item $\mfJ$ is $\CC\oplus\NN$-invariant elementary.
        \item $\mfJ$ is $\epsilon$-invariant elementary.
        \item $\mfJ$ is $\OO$-invariant elementary.
    \end{bracketenumerate}
\end{proposition}

\begin{proof}
    We know that, by \autoref{lem:epssubOO}, $(2)$ implies $(3)$. It remains to be show, then, that $(1)$ implies $(2)$ and that $(3)$ implies $(1)$. For the former, we proceed so as to use the lemma roughly described in \autoref{rem:rhosubeps}. That is to say, we define $\mathrm{FO}+\epsilon$ formulae $\varphi_C$ and $\varphi_N$ (without parameters, in fact) such that the former defines a circular successor and the latter defines a neighbor selector. $\varphi_N(x,y)$ simply says that the choice function $\epsilon$ chooses $y$ from among the neighbors of $x$ in $\mcG(\mcA)$ (as $S_1^\mcA(x)$ is definable). The latter abuses the fact that there are at most $d$ elements of $S_1^\mcA(x)$. An auxiliary distinguishes via $\epsilon$ some element $x_1$ of $S_1^\mcA(x)$, then distinguishes an element $x_2$ of $S_1^\mcA(x)\setminus \{x_1\}$, then an element $x_3$ of $S_1^\mcA(x)\setminus \{x_1,x_2\}$, and so on up to $x_d$ (or until $x_i$ where there are $x_i$ elements in $S_1^\mcA(x)$). $\varphi_C$ then says that $x_{k+1}$ is the successor of $x_k$ and that $x_1$ is the successor of $x_i$ where $i$ is the number of elements of $S_1^\mcA(x)$.
    
    For $(3)$ implies $(1)$, we proceed so as to use the variant of \autoref{lem:elemcapt} described in \autoref{rem:rhosubeps}. The neighbor selector will define the minimal element of $S_1^\mcA(x)$ with respect to the local order, and because there are at most $d$ elements of $S_1^\mcA(x)$, we can enumerate them and recover a linear ordering from the circular successor. In slightly more detail for the latter, we can do a similar construction to $\varphi_C$ in that the neighbor selector distinguishes some $x_1$, and the local circular successor $C$ distinguishes an $x_{k+1}$ for any $k$ (with $k$ less than the number of elements of $S_1^\mcA(x)$). From this labeling, which can consist of at most $d$ elements, we get a definable local ordering where $x_i$ is less than $x_j$ if and only if $i$ is less than $j$.
\end{proof}

The reader may have noted that simply containing the $\CC\oplus\NN$-invariant elementary queries in the $\OO$-invariant elementary ones would be sufficient, since trivially every elementary query is an $\PP$-invariant or $\epsilon$-invariant elementary query. However, the example illustrates several important and generalizable points. 

The first is that we are able to decompose $\OO$ into $\CC\oplus \NN$ on classes of bounded degree. This tells us that other classes may have such decompositions, such as $\LL$, or the class of traversals $\TT$. As we will see later, combining simpler presentation schemes in this way can give us an upper bound on the expressivity of epsilon-invariant logic in terms of presentation schemes. 

The second important principle, which the example does not directly illustrate, is that decompositions like $\OO$ into $\CC\oplus \NN$ can make it easier to use more traditional arguments for invariant logics, especially to play games. Often when one wishes to show that every $\PP$-invariant elementary query is elementary on some class of structures (classes of bounded degree, say, or of bounded treewidth), one will show that for structures $\mcA$ and $\mcB$ such that $\mcA$ is equivalent to $\mcB$ with respect to $\mathrm{FO}[f(k)]$ sentences with $f$ some function $\omega$ to $\omega$, then $\mcA$ and $\mcB$ have presentations $\mcA',\mcB'\in\PP$ such that $\mcA'$ and $\mcB'$ are equivalent with respect to $\mathrm{FO}[k]$. This is the technique of \cite{DBLP:journals/lmcs/Grange21} to show that successor-invariant logic collapses to first-order logic on classes of bounded degree, the technique of \cite{DBLP:conf/csl/Grange23} to show that two-variable order-invariant logic collapses to first-order logic on classes of bounded degree, and so on. If $\mcK\subseteq \mfF_\sigma$ is of bounded degree, then the class of $C$-presentations $(\mcA,C)$ with $\mcA\in\mfK$ and $C$ a local circular successor on $\mcA$ is also a class of bounded degree. This preservation of degree-boundedness also holds for neighbor selectors. Thus, if we can use the technique above for each of $\CC$ and $\NN$ independently on classes of bounded degree, then the technique can be used for their join (thus bounding $\OO$). This is significant because these structures are much more flexible than a full local order. It is possible that a similar technique could be used to bound invariant logics on, say, classes of bounded tree-width.

We now introduce the presentation schemes which we will combine to upper-bound epsilon-invariant logic, and show that this is indeed an upper bound.

\begin{definition}
    Let $Z$ be a $k+1$-ary relation on $\mcA\in\mfF_\sigma$. We say that $Z$ is a $k,r,\varphi$ selector if for every list $x_1,\dots x_k$ of not necessarily distinct elements of $\mcA$, there is a unique $y$ such that $Z(x_1,\dots x_k,y)$. Also, $y$ is not in $B_r^\mcA(x_i)$ for any $i$, and $\varphi(y)$ holds in $\mcA$. We let $\ZZ_{k,r,\varphi}$ be the class of $Z$-presentations which are $k,r,\varphi$ selectors.
\end{definition}

\begin{proposition}
    Let $\mfJ$ be an $\epsilon_k$-invariant elementary boolean query on $\mfK\subseteq\mfF_\sigma$. Then $\mfJ$ is $\OO\oplus \bigoplus_{\varphi\in\mathrm{FO}[k]}\ZZ_{k,r_{k},\varphi}$-invariant elementary.
\end{proposition}

\begin{proof}
    The proof of this fact is very similar to that of \autoref{lem:epssubOO}. The difference is that we cannot fix the sequence of elements with neighborhood types as parameters, both because there are in general infinitely many $r_k$-neighborhood types and because the radius $r_k$ balls around a fixed number of elements might cover arbitrarily many elements of a given neighborhood type. Thus, we use the $k,r_k,\varphi$ selectors to distinguish elements outside of the relevant balls instead.
\end{proof}

Of course, in this current form, we have something fairly unwieldy. However, on classes of bounded tree-width or planar graphs, it stands to reason that this can be simplified and used in a way analogous to $\CC\oplus \NN$ on classes of bounded degree.

\section{Conclusions}

We have established the collapses of local-order-invariant and epsilon-invariant logic on classes of bounded degree, as well as answered several questions of Weinstein. We have also introduced a way to decompose presentation schemes into simpler ones. Interesting avenues for further investigation include the following.
\begin{enumerate}
    \item Does local-order-invariant logic collapse on other tame graph classes? One way to study this might be decompositions as above, but another might be to extend \autoref{thm:localinvhanf} to said tame graph classes instead.
    \item Can upper bounds such as these on epsilon-invariant logic provide us with separation results? It has long been open whether or not epsilon-invariant logic is equal in expressibility to order invariant logic. Could we, for instance, answer this question by demonstrating that order-invariant logic does not collapse to first-order logic on classes of bounded degree? 
    \item Can we find decompositions of order- or traversal-invariant logic which might help us prove upper bounds on those?
\end{enumerate}



\bibliography{blah.bib}

@article{DBLP:journals/jsyml/BenediktS09,
  author       = {Michael Benedikt and
                  Luc Segoufin},
  title        = {Towards a characterization of order-invariant queries over tame graphs},
  journal      = {J. Symb. Log.},
  volume       = {74},
  number       = {1},
  pages        = {168--186},
  year         = {2009},
  doi          = {10.2178/JSL/1231082307},
  bibsource    = {dblp computer science bibliography, https://dblp.org}
}

@Article{dawar_et_al:DagRep.7.9.1,
  author =	{Dawar, Anuj and Gr\"{a}del, Erich and Kolaitis, Phokion G. and Schwentick, Thomas},
  title =	{{Finite and Algorithmic Model Theory (Dagstuhl Seminar 17361)}},
  pages =	{1--25},
  journal =	{Dagstuhl Reports},
  ISSN =	{2192-5283},
  year =	{2018},
  volume =	{7},
  number =	{9},
  editor =	{Dawar, Anuj and Gr\"{a}del, Erich and Kolaitis, Phokion G. and Schwentick, Thomas},
  publisher =	{Schloss Dagstuhl -- Leibniz-Zentrum f{\"u}r Informatik},
  address =	{Dagstuhl, Germany},
  doi =		{10.4230/DagRep.7.9.1},
}

@article{DBLP:journals/iandc/FaginSV95,
  author       = {Ronald Fagin and
                  Larry J. Stockmeyer and
                  Moshe Y. Vardi},
  title        = {On Monadic {NP} vs. Monadic co-{NP}},
  journal      = {Inf. Comput.},
  volume       = {120},
  number       = {1},
  pages        = {78--92},
  year         = {1995},
  doi          = {10.1006/INCO.1995.1100},
  bibsource    = {dblp computer science bibliography, https://dblp.org}
}

@incollection{conf/lc/Gaifman82,
  title = {On Local and Non-Local Properties},
  editor = {J. Stern},
  series = {Studies in Logic and the Foundations of Mathematics},
  publisher = {Elsevier},
  volume = {107},
  pages = {105--135},
  year = {1982},
  booktitle = {Proceedings of the Herbrand Symposium},
  doi = {https://doi.org/10.1016/S0049-237X(08)71879-2},
  author = {Haim Gaifman},
}

@inproceedings{DBLP:conf/mfcs/GroheS98,
  author       = {Martin Grohe and
                  Thomas Schwentick},
  editor       = {Lubos Brim and
                  Jozef Gruska and
                  Jir{\'{\i}} Zlatuska},
  title        = {Locality of Order-Invariant First-Order Formulas},
  booktitle    = {Mathematical Foundations of Computer Science 1998, 23rd International
                  Symposium, MFCS'98, Brno, Czech Republic, August 24-28, 1998, Proceedings},
  series       = {Lecture Notes in Computer Science},
  volume       = {1450},
  pages        = {437--445},
  publisher    = {Springer},
  year         = {1998},
  doi          = {10.1007/BFB0055793},
  bibsource    = {dblp computer science bibliography, https://dblp.org}
}

@incollection{conf/mt/Hanf65,
  title = {Model-theoretic methods in the study of elementary logic},
  editor = {J.W. Addison and Leon Henkin and Alfred Tarski},
  booktitle = {The Theory of Models},
  publisher = {North-Holland},
  pages = {132--145},
  year = {2014},
  series = {Studies in Logic and the Foundations of Mathematics},
  doi = {https://doi.org/10.1016/B978-0-7204-2233-7.50020-4},
  author = {William Hanf},
}

@article{DBLP:journals/jsyml/HellaLN99,
  author       = {Lauri Hella and
                  Leonid Libkin and
                  Juha Nurmonen},
  title        = {Notions of Locality and Their Logical Characterizations over Finite
                  Models},
  journal      = {J. Symb. Log.},
  volume       = {64},
  number       = {4},
  pages        = {1751--1773},
  year         = {1999},
  doi          = {10.2307/2586810},
  bibsource    = {dblp computer science bibliography, https://dblp.org}
}

@book{DBLP:books/sp/Libkin04,
  author       = {Leonid Libkin},
  title        = {Elements of Finite Model Theory},
  series       = {Texts in Theoretical Computer Science},
  publisher    = {Springer},
  year         = {2004},
  doi          = {10.1007/978-3-662-07003-1},
  isbn         = {3-540-21202-7},
  bibsource    = {dblp computer science bibliography, https://dblp.org}
}

@misc{lindell2025hanflocalityinvariantelementary,
      title={Hanf Locality and Invariant Elementary Definability}, 
      author={Steven Lindell and Henry Towsner and Scott Weinstein},
      year={2025},
      eprint={2507.12450},
      archivePrefix={arXiv},
      primaryClass={math.LO},
}

@article{DBLP:journals/jsyml/Rossman07,
  author       = {Benjamin Rossman},
  title        = {Successor-invariant first-order logic on finite structures},
  journal      = {J. Symb. Log.},
  volume       = {72},
  number       = {2},
  pages        = {601--618},
  year         = {2007},
  doi          = {10.2178/JSL/1185803625},
  bibsource    = {dblp computer science bibliography, https://dblp.org}
}

@inproceedings{DBLP:conf/csl/Grange23,
  author       = {Julien Grange},
  editor       = {Bartek Klin and
                  Elaine Pimentel},
  title        = {Order-Invariance in the Two-Variable Fragment of First-Order Logic},
  booktitle    = {31st {EACSL} Annual Conference on Computer Science Logic, {CSL} 2023,
                  February 13-16, 2023, Warsaw, Poland},
  series       = {LIPIcs},
  volume       = {252},
  pages        = {23:1--23:19},
  publisher    = {Schloss Dagstuhl - Leibniz-Zentrum f{\"{u}}r Informatik},
  year         = {2023},
  url          = {https://doi.org/10.4230/LIPIcs.CSL.2023.23},
  doi          = {10.4230/LIPICS.CSL.2023.23},
  bibsource    = {dblp computer science bibliography, https://dblp.org}
}

@article{DBLP:journals/lmcs/Grange21,
  author       = {Julien Grange},
  title        = {Successor-Invariant First-Order Logic on Classes of Bounded Degree},
  journal      = {Log. Methods Comput. Sci.},
  volume       = {17},
  number       = {3},
  year         = {2021},
  url          = {https://doi.org/10.46298/lmcs-17(3:20)2021},
  doi          = {10.46298/LMCS-17(3:20)2021},
  bibsource    = {dblp computer science bibliography, https://dblp.org}
}

@book{DBLP:books/aw/AbiteboulHV95,
  author       = {Serge Abiteboul and
                  Richard Hull and
                  Victor Vianu},
  title        = {Foundations of Databases},
  publisher    = {Addison-Wesley},
  year         = {1995},
  isbn         = {0-201-53771-0},
  bibsource    = {dblp computer science bibliography, https://dblp.org}
}

@article{DBLP:journals/jsyml/Otto00,
  author       = {Martin Otto},
  title        = {Epsilon-Logic Is More Expressive Than First-Order Logic Over Finite
                  Structures},
  journal      = {J. Symb. Log.},
  volume       = {65},
  number       = {4},
  pages        = {1749--1757},
  year         = {2000},
  doi          = {10.2307/2695073},
  bibsource    = {dblp computer science bibliography, https://dblp.org}
}

@inproceedings{DBLP:conf/birthday/BhaskarLW24,
  author       = {Siddharth Bhaskar and
                  Steven Lindell and
                  Scott Weinstein},
  editor       = {Antoine Amarilli and
                  Alin Deutsch},
  title        = {Traversal-Invariant Characterizations of Logarithmic Space},
  booktitle    = {The Provenance of Elegance in Computation - Essays Dedicated to Val
                  Tannen, Tannen's Festschrift, May 24-25, 2024, University of Pennsylvania,
                  Philadelphia, PA, {USA}},
  series       = {OASIcs},
  volume       = {119},
  pages        = {2:1--2:17},
  publisher    = {Schloss Dagstuhl - Leibniz-Zentrum f{\"{u}}r Informatik},
  year         = {2024},
  url          = {https://doi.org/10.4230/OASIcs.Tannen.2},
  doi          = {10.4230/OASICS.TANNEN.2},
  bibsource    = {dblp computer science bibliography, https://dblp.org}
}

@article{DBLP:journals/corr/HeuvelKPQRS17,
  author       = {Jan van den Heuvel and
                  Stephan Kreutzer and
                  Michal Pilipczuk and
                  Daniel A. Quiroz and
                  Roman Rabinovich and
                  Sebastian Siebertz},
  title        = {Model-Checking for Successor-Invariant First-Order Formulas on Graph
                  Classes of Bounded Expansion},
  journal      = {CoRR},
  volume       = {abs/1701.08516},
  year         = {2017},
  url          = {http://arxiv.org/abs/1701.08516},
  eprinttype    = {arXiv},
  eprint       = {1701.08516},
  bibsource    = {dblp computer science bibliography, https://dblp.org}
}

@inproceedings{DBLP:conf/lics/EngelmannKS12,
  author       = {Viktor Engelmann and
                  Stephan Kreutzer and
                  Sebastian Siebertz},
  title        = {First-Order and Monadic Second-Order Model-Checking on Ordered Structures},
  booktitle    = {Proceedings of the 27th Annual {IEEE} Symposium on Logic in Computer
                  Science, {LICS} 2012, Dubrovnik, Croatia, June 25-28, 2012},
  pages        = {275--284},
  publisher    = {{IEEE} Computer Society},
  year         = {2012},
  url          = {https://doi.org/10.1109/LICS.2012.38},
  doi          = {10.1109/LICS.2012.38},
  bibsource    = {dblp computer science bibliography, https://dblp.org}
}

@inproceedings{DBLP:conf/csl/GrangeS20,
  author       = {Julien Grange and
                  Luc Segoufin},
  editor       = {Maribel Fern{\'{a}}ndez and
                  Anca Muscholl},
  title        = {Order-Invariant First-Order Logic over Hollow Trees},
  booktitle    = {28th {EACSL} Annual Conference on Computer Science Logic, {CSL} 2020,
                  January 13-16, 2020, Barcelona, Spain},
  series       = {LIPIcs},
  volume       = {152},
  pages        = {23:1--23:16},
  publisher    = {Schloss Dagstuhl - Leibniz-Zentrum f{\"{u}}r Informatik},
  year         = {2020},
  url          = {https://doi.org/10.4230/LIPIcs.CSL.2020.23},
  doi          = {10.4230/LIPICS.CSL.2020.23},
  bibsource    = {dblp computer science bibliography, https://dblp.org}
}

@article{DBLP:journals/corr/BarceloL16,
  author       = {Pablo Barcel{\'{o}} and
                  Leonid Libkin},
  title        = {Order-Invariant Types and Their Applications},
  journal      = {Log. Methods Comput. Sci.},
  volume       = {12},
  number       = {1},
  year         = {2016},
  url          = {https://doi.org/10.2168/LMCS-12(1:9)2016},
  doi          = {10.2168/LMCS-12(1:9)2016},
  bibsource    = {dblp computer science bibliography, https://dblp.org}
}

@article{DBLP:journals/corr/HarwathS16,
  author       = {Frederik Harwath and
                  Nicole Schweikardt},
  title        = {On the locality of arb-invariant first-order formulas with modulo
                  counting quantifiers},
  journal      = {Log. Methods Comput. Sci.},
  volume       = {12},
  number       = {4},
  year         = {2016},
  url          = {https://doi.org/10.2168/LMCS-12(4:8)2016},
  doi          = {10.2168/LMCS-12(4:8)2016},
  bibsource    = {dblp computer science bibliography, https://dblp.org}
}

@inproceedings{DBLP:conf/pods/AbiteboulSV90,
  author       = {Serge Abiteboul and
                  Eric Simon and
                  Victor Vianu},
  editor       = {Daniel J. Rosenkrantz and
                  Yehoshua Sagiv},
  title        = {Non-Deterministic Languages to Express Deterministic Transformations},
  booktitle    = {Proceedings of the Ninth {ACM} {SIGACT-SIGMOD-SIGART} Symposium on
                  Principles of Database Systems, April 2-4, 1990, Nashville, Tennessee,
                  {USA}},
  pages        = {218--229},
  publisher    = {{ACM} Press},
  year         = {1990},
  url          = {https://doi.org/10.1145/298514.298575},
  doi          = {10.1145/298514.298575},
  bibsource    = {dblp computer science bibliography, https://dblp.org}
}

@article{DBLP:journals/amai/AbiteboulV91,
  author       = {Serge Abiteboul and
                  Victor Vianu},
  title        = {Non-Determinism in Logic-Based Languages},
  journal      = {Ann. Math. Artif. Intell.},
  volume       = {3},
  number       = {2-4},
  pages        = {151--186},
  year         = {1991},
  url          = {https://doi.org/10.1007/BF01530924},
  doi          = {10.1007/BF01530924},
  bibsource    = {dblp computer science bibliography, https://dblp.org}
}

@article{DBLP:journals/tcs/Courcelle96,
  author       = {Bruno Courcelle},
  title        = {The Monadic Second-Order Logic of Graphs {X:} Linear Orderings},
  journal      = {Theor. Comput. Sci.},
  volume       = {160},
  number       = {1{\&}2},
  pages        = {87--143},
  year         = {1996},
  url          = {https://doi.org/10.1016/0304-3975(95)00083-6},
  doi          = {10.1016/0304-3975(95)00083-6},
  bibsource    = {dblp computer science bibliography, https://dblp.org}
}

@article{Hilbert22,
author = {David Hilbert},
title = {Neubegründung der Mathematik. Erste Mitteilung},
journal = {Abhandlungen aus dem Mathematischen Seminar der Universität Hamburg},
volume = {1},
pages = {15--177},
year = {1922},
url = {https://doi.org/10.1007/BF02940589},
doi = {10.1007/BF02940589}

}

@incollection{conf/cc/Fagin74,
  author = {Ronald Fagin},
  title = {Generalized First-Order Spectra and Polynomial Time Recognizable Sets},
  editor = {Richard M. Karp},
  booktitle = {Complexity of Computation. Proceedings of a Symposium in Applied Mathematics of the Americal Mathematical Society and the Society for Industrial and Applied Mathematics held in New York City April 18-19, 1973},
  publisher = {AMS},
  pages = {43--73},
  year = {1974},
  series = {SIAM-AMS Proceedings},
  volume = {7},
  isbn = {0-8218-1327-7},
}

\appendix
\section{Section 3 Details}

\begin{proof}[Proof of \autoref{lem:elempres}]\label{proof:elempres}
    First, we demonstrate the straightforward definition of $\varphi_{LO,\psi}(R,x)$ mentioned earlier:
    \begin{align*}
    \varphi_{LO,\psi}(R,\vec x)\equiv (\forall x,y,z) (\psi(\vec x,x)\wedge\psi(\vec x,y)\wedge\psi(\vec x,z))\rightarrow& \bigg(R(x,x) \wedge \\
    &(R(x,y)\wedge R(y,x)\rightarrow x=y) \wedge \\
    & (R(x,y)\vee R(y,x)) \wedge \\
    & (R(x,y)\wedge R(y,z)\rightarrow R(x,z))\bigg).
\end{align*}
    
    The construction of $\beta_r$ will be done inductively, but first, some setup. Let $V_n$ denote $\{x,y,x_1,\dots x_n\}$, and let $Z_n$ denote the set of tuples $(z_1,\dots z_{m})$ for $m\leq n$, where each $z_i$ is in $V_n$ and there are $j,k$ such that $z_j=x$ and $z_k=y$. If for every $R$ in $\sigma$, the arity of $R$ is $a_R$, and $a_{max}=\max(\{a_R:R\in\sigma\})$, then we can define the edge relationship $\eta$ in $\mcG(\mcA)$ as follows:
    \begin{align*}
        \eta(x,y)\equiv x=y\vee (\exists x_1,\dots x_{a_{max}})\bigvee_{\substack{R\in \sigma\\ \vec z \in Z_{a_R}}} R(\vec z).
    \end{align*}
    From here, defining $\beta_r$ is simply the statement that there is a path in $\mcG(\mcA)$ of length at most $r$, which can be stated as follows:
    \begin{align*}
        \beta_r(x,y)\equiv \bigvee_{r'\leq r} (\exists x_1,\dots x_{r'-1}) \eta(x,x_1)\wedge \eta(x_{r-1},y)\wedge\bigwedge_{i\leq r'-2} \eta(x_i,x_{i+1}).
    \end{align*}
    We will say that $\beta_0(x,y)$ will be defined as $x=y$. Now, strictly speaking, the full generality of the definition of $\beta_r$ was not necessary; we could have simply defined $\beta_1$ for the purposes of the lemma. However, this definition will be useful for later proofs, so we put it here. We can also define, from this, $\phi_r(x,y)\equiv \beta_r(x,y)\wedge\neg \beta_{r-1}(x,y)$, which says that $y$ is in $S_r^\mcA(x)$. We now finally define the sentence which says that $\preceq$ is a local order:
    \begin{align*}
        (\forall x,y,z,w) (\neg\phi_1(x,y)\vee \neg\phi_1(x,z)\rightarrow \neg\preceq(x,y,z))\wedge \varphi_{LO,\phi_1}(\preceq,x).
    \end{align*}
\end{proof}

\section{Section 4 Details}

\begin{proof}[Proof of \autoref{lem:elemcapt}]\label{proof:elemcapt}
    We now explain this in more detail. Let $\theta$ be a first-order sentence in the language of $\sigma\cup\{R\}$. Then converting to prefix normal form and then conjunctive normal form yields a sentence
    \begin{align*}
        \theta'\equiv (Q_1y_1)\dots (Q_nxy_n)\bigwedge_{i\leq j}\bigvee_{k\leq l}\phi_{i,k}(y_1,\dots y_n),
    \end{align*}
    where each $Q_m$ is an existential or universal quantifier and each $\phi_{i,k}$ is either of the form $T_{i,k}(\vec {z})$ or $\neg T_{i,k}(\vec z)$, where $T_{i,k}$ is a relation in $\sigma\cup\{R\}$ with arity $a_{T_{i,k}}$ and $\vec z$ is a tuple in the set $\{y_1,\dots y_n\}^{a_{T_{i,k}}}$. For each $\phi_{i,k}$, define $\phi'_{i,k}$ such that
    \begin{enumerate}
        \item if $T_{i,k}$ is $R$, then if $\phi_{i,k}$ is of the form $T(\vec z)$ or $\neg T(\vec z)$, $\phi'(x_1,\dots x_n,\vec z)$ is the formula $\varphi(\vec x,\vec z)$ or $\neg \varphi(\vec x,\vec z)$ respectively, and
        \item if $T_{i,k}$ is in $\sigma$, then $\phi'_{i,k}$ is simply $\phi_{i,k}$.
    \end{enumerate}
    For sufficient $\vec x$, $\varphi(\vec x,\vec z)$ defines a relation such that $(A,\varphi(\vec x,\vec z))\in\PP$, so for sufficient choice of $\vec x$, $\theta'$ is equivalent to
    \begin{align*}
        \theta''(\vec x)\equiv (Q_1 y_1)\dots(Q_ny_n)\bigwedge_{i\leq j}\bigvee_{k\leq l} \phi_{i,k}'(\vec x,y_1,\dots y_n).
    \end{align*}
    Finally, if $\vec x$ has length $m$, the $\PP$-invariant formula equivalent to $\theta$ is
    \begin{align*}
        \theta'''\equiv (\forall x_1,\dots x_m) \varphi'(x_1,\dots x_m)\rightarrow \theta''(x_1,\dots x_m).
    \end{align*}
\end{proof}

\begin{proof}[Proof of \autoref{lem:bounddiamcapt}]\label{proof:bounddiamcapt}
    We demonstrate that $\lambda_r$ is first-order definable. We have already shown how to define $\lambda_1$. From $\lambda_r$ we will define $\lambda_{r+1}'(x,y,z)$ as follows:
        \begin{align*}
            \lambda_{r+1}'(x,y,z)\equiv (\exists x_1)(\forall x_2)\beta_r(x,x_1)\wedge \eta(x_1,y)\wedge (\beta_r(x,x_2)\wedge\eta(x_2,z)\rightarrow \lambda_r(x,x_1,x_2)).
        \end{align*}
        This breaks in some ways when $y$ and $z$ are in $B_r^\mcA(x)$, but this does not matter, because we will only use it when it is guaranteed that $y$ and $z$ are in $S_{r+1}^\mcA(x)$. We will define several cases, of which $\lambda_{r+1}$ will be a disjunction:
        \begin{align*}
            \lambda_{r+1}^1(x,y,z)\equiv \beta_r(x,y)\wedge\beta_r(x,z)\wedge \lambda_r(x,y,z).
        \end{align*}
        This says that both $y$ and $z$ are in the $B_r^\mcA(x)$ initial segment, and as mentioned, are ordered based on that. We also have the following:
        \begin{align*}
            \lambda_{r+1}^2(x,y,z)\equiv \beta_r(x,y)\wedge \phi_{r+1}(x,z),
        \end{align*}
        which says that $y$ is in the initial segment and $z$ is not, but that $z$ is in $S_{r+1}^\mcA(x)$. Before we cover our final individual case, we introduce a part of it, which says that if $x_1$ is the least element which $y$ and $z$ are connected to (in the case where they are equivalent with respect to $\lambda_{r+1}'$), $y$ is less than $z$ with respect to $\preceq(x_1,y,z)$:
        \begin{align*}
            \lambda_{r+1}^{3'}(x,y,z)\equiv (\exists x_1)(\forall x_2) &\beta_r(x,x_1)\wedge\eta(x_1,y)\wedge\eta(x_1,z) \\
            &\wedge ((\beta_r(x,x_2)\wedge (\eta(x_2,y)\wedge\eta(x_2,z)))\rightarrow \lambda_r(x,x_1,x_2)).
        \end{align*}
        Our final individual case, which defines the ordering of $S_{r+1}^\mcA(x)$ as previously described:
        \begin{align*}
            \lambda_{r+1}^3(x,y,z)\equiv &\phi_{r+1}(x,y) \wedge \phi_{r+1}(x,z) \\
            &\wedge \lambda_r'(x,y,z)\wedge (\neg\lambda_r'(x,z,y)\vee (\lambda_r'(x,z,y)\wedge \lambda_{r+1}^{3'}(x,y,z))).
        \end{align*}
        We finish with the definition of $\lambda_{r+1}(x,y,z)$:
        \begin{align*}
            \lambda_{r+1}(x,y,z)\equiv \lambda_{r+1}^1(x,y,z)\vee \lambda_{r+1}^2(x,y,z)\vee\lambda_{r+1}^3(x,y,z).
        \end{align*}
        Finally, by the bounded diameter of $\mfK$, there is some $d$ such that for every $\mcA\in\mfK$ and every $x$ in $\mcA$, $\lambda_{d}(x,y,z)$ defines a linear ordering on $\mcA$. Thus, by \autoref{lem:elemcapt}, we have the theorem.
\end{proof}

\begin{proof}[Proof of \autoref{ex:applications}]\label{proof:applications}
    Fix $k$. For any $S_k$ which is the graph of a permutation, the following auxiliary formula says (for natural $i$) that a length $k$ tuple $(a_1,\dots a_k)$ satisfies $a_m\neq a_n$ for $m\lneq n\leq i$, and that $a_m=a_n$ for $j,k\gneq i$:
        \begin{align*}
            \varphi_i'(a_1,\dots a_k)\equiv  \left(\bigwedge_{m\lneq n\leq i} a_m\neq a_n\right) \wedge \left(\bigwedge_{m,n\geq i} a_m= a_n\right). 
        \end{align*}
        The proper $\varphi'$ which we will invoke for the application of \autoref{lem:elemcapt} is this one, which says that each of the $a_i$'s is a fixed point, that the tuple follows the $\varphi_i'$ structure for some $i$, and that $b$ is not a fixed point:
        \begin{align*}
            \varphi'(a_1,\dots a_k,b)\equiv (\forall x)&\left(S_k(x,x)\rightarrow \bigvee_{m\leq k} a_m=x\right)\\&\wedge\left(\bigvee_{i\leq k}\varphi_i'(a_1,\dots a_k)\wedge \bigwedge_{m\leq i}S(a_m,a_m)\right) \\&\wedge ((\exists x)\neg S_k(x,x) \rightarrow \neg S_k(b,b)).
        \end{align*}
        Now, we define the $\varphi$ which completes our application of the lemma:
        \begin{align}
            \varphi(a_1,\dots a_k,b, x,y)\equiv& ([\varphi_0'(a_1,\dots a_k)\vee(x\neq b\wedge\neg S_k(b,y)\wedge x\neq y)]\wedge S_k(x,y))\\
            &\vee \bigvee_{\substack{1\leq i\leq k \\ m\lneq i}} \bigg(\varphi_i'(a_1,\dots a_k)\\
            &\wedge [(x=b\wedge y=a_1)\vee (x=a_i\wedge S_k(b,y))\\
            &\vee (x=a_m\wedge y=a_{m+1}\wedge a_m\neq a_{m+1})]\bigg).
        \end{align}
        Informally, this definition takes any fixed points and inserts them between $b$ and its successor.
\end{proof}

\section{Section 5 Details}

\begin{proof}[Proof of \autoref{lem:epscontained}]\label{proof:epscontained}
    For any $\epsilon$-invariant sentence $\theta\equiv (Q_1x_1)\dots (Q_nx_n)\theta'(x_1,\dots x_n)$ with $\theta'$ quantifier-free, we replace every $\epsilon$ term $\epsilon_x(\varphi(\vec x, y))$ in $\theta'$ with a variable $x_\varphi$, yielding a formula $\theta''(x_1,\dots x_n,\vec {x_\varphi},\vec z)$, where $\vec {x_\varphi}$ is a tuple containing $x_\varphi$ for every $\epsilon_k$ argument $\varphi(\vec x, y)$. Enumerating the possible arguments $\varphi_1,\dots \varphi_m$ (that there are finitely many up to equivalence is a folklore result, see \cite{DBLP:books/sp/Libkin04}), we write for our equivalent $\PP$-invariant sentence 
    \begin{align*}
        \theta''\equiv &(\forall z_1)\dots (\forall z_l)(Q_1x_1)\dots (Q_nx_n)(\forall x_{\varphi_1})\dots (\forall x_{\varphi_m}) \\
        &\psi(z_1,\dots z_l)\wedge \left(\bigwedge_{i\leq m} \Phi_{\varphi_i}(x_1,\dots x_n,x_{\varphi_i},z_1,\dots z_l)\right)
         \\
         &\rightarrow \theta''(x_1,\dots x_n,x_{\varphi_1},\dots x_{\varphi_m},z_1,\dots z_l).
    \end{align*}
    This suffices because the $x_\varphi$'s are chosen for every $x_1,\dots x_n$ only by the $\PP$ structure and the parameters $z_1,\dots z_l$.
\end{proof}

\begin{proof}[Proof of \autoref{lem:epssubOO}]\label{proof:epssubOO}
    We begin so as to make use of \autoref{lem:epscontained}, showing instead that every $\epsilon_k$-invariant elementary boolean query is $\OO$-invariant elementary. We make strong use of \autoref{thm:gaifmanlocality}. There are finitely many $\mathrm{FO[k,k+1]}$ formulae $\varphi(x,\vec y)$ up to equivalence. For fixed $\vec y$, each of these is determined by the $r$-neighborhood type of $x$ in the structure $(\mcA,\vec y)$ (i.e., with each element of $\vec y$ given a constant symbol) for some $r$. 
    
    Thus, there is some fixed $r_k$ such that each $\varphi$ depends only on the $r_k$-neighborhood of $x$ in these pointed structures. By degree-boundedness, there are finitely many such neighborhood types, which we label $\tau_1,\dots \tau_n$, so each $\varphi$ is a boolean combination of the $\varphi_{\tau_i}$ which define such neighborhood types. Thus, we can restrict our attention to defining choice functions for the $r_k$-neighborhood types. 
    
    For any neighborhood type $\tau_i$ of the pointed structures described above, either $\tau_i$ contains one of the constants or it does not. If it does, then the linear order we will now describe lets us select in a definable way a minimum element with that neighborhood type. $\vec y$ has length at most $k$, so a local order $\preceq$ defines a linear order on $B_{r_k}^\mcA(y_1)\cup \dots B_{r_k}^\mcA(y_i)$ where $i$ is the length of $\vec y$. Specifically, this linear order has the elements of $A_j=B_{r_k}^\mcA(y_1)\cup \dots B_{r_k}^\mcA(y_{j})$  precede the elements of $B_{r_k}^\mcA(y_{j+1})\setminus A_j$, and $\lambda_{r_k}(x,y,z)$ orders each $A_j$ internally. 
    
    We now explain the procedure when $\tau_i$ does not contain one of the constants. We define $\psi(\vec z)$, where $\vec z$ is a tuple including a variable $z_{j,j'}$ for every $j\leq d\cdot n(d,r_k)$ and $j'\leq n$ (recall that a neighborhood of radius $r$ in a structure of degree at most $d$ has cardinality at most $n(d,r)+1$). $\psi$ says that for every $j'$, either $z_{j,j'}$ is distinct from $z_{j'',j'}$ for every $j$ and $j'$ or every element with $r_k$-neighborhood type $\tau_{j'}$ is equal to some $z_{j,j'}$. For each $j'$, we then can take the least $j$ such that $z_{j,j'}$ is not in the $r_k$-neighborhoods of any of the constants, and select $z_{j,j'}$ as the distinguished element of type $\tau_{j'}$. Note that such a $z_{j,j'}$ must exist by the Pigeonhole Principle.
\end{proof}

\end{document}